\documentclass[a4paper,12pt]{article}

\usepackage[top=2.54cm,bottom=2.54cm,left=2.54cm,right=2.54cm]{geometry}
\usepackage{times}
\usepackage{lmodern}
\usepackage[english]{babel}
\usepackage{amsmath,amssymb,mathtools,amsthm}
\usepackage{graphicx}
\usepackage{booktabs}
\usepackage{array}
\usepackage{tabularx}
\usepackage{multirow}
\usepackage[labelfont=bf,textfont=bf]{caption}
\usepackage{enumitem}
\usepackage{natbib}
\usepackage{titlesec}
\usepackage{xcolor}
\usepackage{hyperref}

\hypersetup{
  colorlinks=true,
  citecolor=blue,
  linkcolor=blue,
  urlcolor=blue,
  pdftitle={Finite-sample certification and operating envelopes for spectral clustering and graph centrality},
  pdfauthor={Chandrasekhar Gokavarapu, Sekhar Babu Gosala, Rajeev Muthu, Vamsi Pasalapudi, Tarakarama Kapakayala},
  pdfkeywords={spectral clustering, graph centrality, confidence regions, matrix Bernstein inequality}
}
\titleformat{\section}[block]{\bfseries\raggedright}
{\makebox[1.27cm][l]{\thesection.}}{0in}{}
\titlespacing*{\section}{0pt}{3mm}{0pt}
\titleformat{\subsection}[block]{\bfseries}
{\makebox[1.27cm][l]{\thesubsection.}}{0in}{}
\titlespacing*{\subsection}{0pt}{3mm}{0pt}
\titleformat{\subsubsection}[block]{\bfseries}
{\makebox[1.27cm][l]{\thesubsubsection.}}{0in}{}
\titlespacing*{\subsubsection}{0pt}{3mm}{0pt}

\newtheoremstyle{sa}{3mm}{2mm}{}{}{\bfseries}{:}{.5em}{}
\theoremstyle{sa}
\newtheorem{theorem}{Theorem}
\newtheorem{corollary}[theorem]{Corollary}
\newtheorem{lemma}[theorem]{Lemma}
\newtheorem{proposition}[theorem]{Proposition}

\newtheorem{remark}[theorem]{Remark}

\renewenvironment{proof}{{\noindent\bfseries Proof:}}{\hfill$\square$}

\newcommand{\R}{\mathbb{R}}
\newcommand{\E}{\mathbb{E}}
\newcommand{\PP}{\mathbb{P}}
\newcommand{\1}{\mathbf{1}}
\newcommand{\norm}[1]{\left\lVert#1\right\rVert}

\newcommand{\Gr}{\operatorname{Gr}}
\newcommand{\dgr}{d_{\mathrm{Gr}}}

\newcommand{\diam}{\operatorname{diam}}
\newcommand{\TV}{\operatorname{TV}}
\newcommand{\KL}{\operatorname{KL}}

\newcommand{\KarateAccuracy}{0.971}
\newcommand{\KarateHamming}{1}
\newcommand{\KarateBoundaryGap}{2.061}
\newcommand{\KarateOperatorQuantile}{28.119}
\newcommand{\KarateGapLower}{-54.177}

\newcommand{\KarateDegreeHalf}{16.695}
\newcommand{\UnequalDegreeMargin}{499.5}
\newcommand{\UnequalDegreeHalf}{84.1}
\newcommand{\UnequalObservedGap}{404}
\newcommand{\KatzMarginRatio}{2.08}
\newcommand{\CentralityFiring}{0.711}
\newcommand{\CentralityCorrectGivenFire}{1.000}

\newcommand{\ExactRadiusThreshold}{1{,}610}
\newcommand{\DegreeRadiusThreshold}{2{,}106}
\newcommand{\HammingThreshold}{58{,}508}

\title{\bfseries Finite-sample certification and operating envelopes for
spectral clustering and graph centrality}
\author{Chandrasekhar Gokavarapu$^{1}$\thanks{Corresponding author:
\href{mailto:chandrasekhargokavarapu@gmail.com}
{chandrasekhargokavarapu@gmail.com}}, Sekhar Babu Gosala$^{1}$,
Rajeev Muthu$^{1}$,\\
Vamsi Pasalapudi$^{2}$ and Tarakarama Kapakayala$^{2}$\\[2mm]
\small $^{1}$Department of Mathematics, Government College (Autonomous),\\[-1mm]
\small Rajahmundry, Andhra Pradesh, India\\
\small $^{2}$M.Sc. Mathematics Programme, Government College (Autonomous),\\[-1mm]
\small Rajahmundry, Andhra Pradesh, India}
\date{July 2026}

\begin{document}
\maketitle
\vspace{-5mm}

\noindent\rule{\linewidth}{0.8pt}
\vspace{-3mm}

\noindent{\bfseries Abstract}\quad
Spectral clustering and node rankings are commonly reported from one observed
network without a finite-sample statement of what the observation supports.
We develop a certification protocol that either returns a coverage-guaranteed
set or explicitly returns ``no nontrivial certificate.''  For an
inhomogeneous Bernoulli graph, a matrix-Bernstein quantile with all numerical
constants and its ambient-dimension factor retained is combined with a
one-sided spectral-gap certificate.  The resulting Grassmann ball is valid at
finite \(n\), but is reported as informative only when its radius is below the
   diameter of the Grassmannian.  We propagate the ball through a
   certificate-bearing approximate \(k\)-means map under declared population
   separation and minimum-cluster envelopes, derive simultaneous bands and an observed-gap
certificate for degree centrality, and give a corrected normalized-Katz
extension.  A \(12\)-cell simulation study with \(1{,}000\) graphs per cell
maps the difference between coverage and usefulness.  The submitted
\(n=200\) block-model example is shown to be necessarily vacuous after the
dimension factor is restored; in the benchmark \(p=0.30,q=0.10\), the
   subspace radius first falls below one at
   \(n=\ExactRadiusThreshold\), whereas the mean-square clustering certificate
   remains unavailable until \(n=\HammingThreshold\).
An unequal-block example produces a genuine centrality certificate, while an
analysis of the Zachary karate-club network correctly declines to certify
despite \(97.1\%\) agreement with the observed factions.  These results
separate algorithmic success, coverage validity and inferential
informativeness.

\medskip
\noindent{\itshape Key words:} spectral clustering; graph centrality;
confidence regions; matrix Bernstein inequality; stochastic block model;
network bootstrap.

\noindent{\bfseries AMS Subject Classifications:} 62G15; 62H30; 05C80; 15A18.

\noindent\rule{\linewidth}{0.8pt}

\section{Introduction}\label{sec:intro}

An analyst observes one network, computes two eigenvectors of its adjacency
matrix, runs \(k\)-means and reports a partition.  A second analyst ranks the
same vertices by a centrality score.  Both outputs are deterministic functions
of the observed adjacency matrix, but neither is the corresponding population
object under a sampling model.  The inferential question is therefore not
whether an algorithm returned an answer.  It is whether the single observed
network supports a nontrivial set of population answers at a declared
confidence level.

Two obstructions must be handled.  First, eigenspaces are unstable when the
population spectral gap at the selected boundary is small.  Second, an
eigenbasis is defined only up to orthogonal rotation, so uncertainty belongs
on a Grassmannian rather than in the entries of an arbitrarily oriented
eigenvector matrix.  Davis--Kahan perturbation theory addresses the second
point and quantifies the first, but a perturbation inequality is not by itself
a confidence statement \citep{DavisKahan1970RotationIII,YuWangSamworth2015DavisKahanVariant}.
It must be paired with a finite-\(n\) quantile for the random operator error and
with a verifiable lower bound on the population gap.

The present paper builds and audits that pairing.  Matrix Bernstein gives an
explicit operator quantile, but its dimension factor cannot be suppressed:
for an \(n\times n\) self-adjoint sum the tail prefactor is proportional to
\(n\) \citep{Tropp2012UserFriendlyTailBounds}.  Retaining this factor can turn
an apparently informative radius into the entire Grassmannian.  We treat that
outcome as a result rather than as a numerical inconvenience.  The protocol
returns a valid ball and separately records whether its radius is below the
diameter one.  When the answer is negative, the output is ``no nontrivial
certificate.''

This distinction is important because spectral clustering can perform well
even when a uniform finite-sample certificate is too conservative to say so.
Classical theory establishes consistency and recovery in stochastic block
models under increasingly weak conditions
\citep{RoheChatterjeeYu2011SpectralHighDimSBM,LeiRinaldo2015SpectralClusteringSBM,Abbe2018SBMSurvey}.
Adjacency spectral embedding also has refined asymptotic distribution theory
\citep{SussmanTangFishkindPriebe2012ASE,AthreyaEtAl2016RDPGCLT,RubinDelanchyCapeTangPriebe2022GRDPGJRSSB}.
Those results answer different questions from a finite-sample confidence set
constructed from one graph.  Bootstrap methods provide another route, but
their validity depends on how the unobserved probability matrix is estimated;
recent work explicitly tests and improves network bootstraps through
nearest-neighbour smoothing
\citep{DilworthDavisLawson2025ValidBootstrapsNetworkEmbeddings}.

Our contribution is not a new concentration inequality or a new clustering
algorithm.  It is an end-to-end, numerically auditable inferential protocol
with four components:
\begin{enumerate}[leftmargin=18pt,label=(\roman*)]
\item an exact-constant matrix-Bernstein quantile and a coverage theorem for
the top-\(k\) population eigenspace;
\item a fully observable version using a simultaneous upper confidence bound
for the maximum expected degree and a Weyl lower bound for the population
gap;
\item propagation to certificate-bearing approximate \(k\)-means labels under
declared population envelopes, and centrality selections through observed
margins; and
\item an operating-envelope analysis that reports where a valid region is
nontrivial, supported by simulations, bootstrap comparisons and a real
network diagnostic.
\end{enumerate}

The empirical conclusion is deliberately mixed.  Coverage is easy when a ball
equals the whole parameter space; useful coverage is not.  In the balanced
dense block model used below, the analytic eigenspace region eventually
becomes nontrivial, but the Frobenius-to-Hamming propagation needs much larger
\(n\).  A strongly heterogeneous model supports centrality selection, whereas
the small karate-club network returns no certificate even though its spectral
partition nearly reproduces the recorded split.  This is the behaviour a
certificate should have: it must be allowed to decline.

\subsection{Relation to existing inferential approaches}

There are four neighbouring literatures, but none can be inserted as a
black-box justification for the present confidence set.  First, random-matrix
concentration controls \(\norm{A-P}\).  Matrix Bernstein is uniform and
explicit but conservative; graph-specific results sharpen the rate in
moderately sparse regimes and regularisation restores concentration in regimes
where raw adjacency fails
\citep{Tropp2012UserFriendlyTailBounds,LeLevinaVershynin2017ConcentrationRegularization}.
Our use of Bernstein is motivated by auditability: its assumptions, numerical
constant and dimension dependence can all be displayed in one line.

Second, distributional theory for spectral projectors and embeddings provides
normal approximations, central limit theorems and bootstrap approximations
\citep{KoltchinskiiLounici2017SpectralProjectorsAOS,
JirakWahl2024BootstrapSpectralProjectors,AthreyaEtAl2016RDPGCLT}.
These results can be substantially sharper than a uniform norm ball.  They
also target specific asymptotic regimes and require their own approximation
errors to be controlled before they become finite-sample confidence regions.
The present paper does not treat an asymptotic quantile as exact at the
observed \(n\).

Third, network bootstrapping must estimate a probability matrix from one
adjacency matrix.  Different smoothers can yield plausible but statistically
distinguishable bootstrap graphs.  The validation scheme and ASE-\(k\)NN
smoother of
\citet{DilworthDavisLawson2025ValidBootstrapsNetworkEmbeddings} directly
address this issue.  We use a simplified version of that smoother as an
empirical comparator, without implementing or claiming its exchangeability
validation step, and not as a replacement for the analytic coverage theorem.

Post-selection work for a single observed network instead uses data thinning
or splitting to conduct inference on selected mean connectivities
\citep{AncellWittenKessler2026PostSelectionInferenceNetwork}.  Its target and
randomisation differ from the present eigenspace ball.  Likewise, the
nonparametric spectral-clustering bootstrap of
\citet{WelshShreeves2022NonparametricBootstrapSpectralClustering} is developed
for Euclidean observations and mixture-model clustering, not for estimating a
Bernoulli edge-probability matrix from one network.  We therefore do not treat
either method as an interchangeable network-bootstrap comparator.

Fourth, the strongest community-detection theory often controls individual
rows or misclassification directly rather than passing through a uniform
Grassmann ball.  Approximate \(k\)-means transfer
\citep{LeiRinaldo2015SpectralClusteringSBM} and modern
\(\ell_p\) eigenvector perturbation \citep{AbbeFanWang2022LpPCASpectralClustering}
can therefore prove recovery where our mean-square Hamming set is still
trivial.  This is expected: a general confidence region pays for uniformity
over every admissible perturbation, whereas a model-specific recovery theorem
uses more structure.

\begin{table}[t]
\centering
\caption{Different uncertainty questions for one observed network.}
\label{tab:approaches}
\footnotesize
\setlength{\tabcolsep}{4pt}
\begin{tabularx}{\textwidth}{
  >{\raggedright\arraybackslash}p{0.19\textwidth}
  >{\raggedright\arraybackslash}p{0.22\textwidth}
  >{\raggedright\arraybackslash}p{0.19\textwidth}
  >{\raggedright\arraybackslash}X}
\toprule
Approach&Primary target&Calibration&Typical output\\
\midrule
Uniform operator certificate&Population eigenspace&Finite-sample analytic
tail&Coverage ball or no nontrivial certificate\\
Embedding CLT/projector approximation&Local coordinates or projector
functionals&Asymptotic approximation&Standard errors or approximate
quantiles\\
Network bootstrap&Distribution induced by an estimated \(P\)&Empirical or
method-specific validity&Bootstrap cloud or radius\\
Recovery theorem&Labels under a structured model&High-probability model
bound&Error rate or exact recovery\\
\bottomrule
\end{tabularx}
\end{table}

\section{Model, geometry and an explicit operator quantile}\label{sec:model}

\subsection{Sampling model and target}

Let \(A\in\{0,1\}^{n\times n}\) be symmetric with zero diagonal.  Conditional
on a fixed probability array \(P=(P_{ij})\), the variables
\(\{A_{ij}:i<j\}\) are independent and
\[
 A_{ij}\sim\operatorname{Bernoulli}(P_{ij}),\qquad
 A_{ji}=A_{ij},\qquad A_{ii}=P_{ii}=0.
\tag{2.1}\label{eq:model}
\]
An SBM is specified off the diagonal by
\(P_{ij}=B_{g_i g_j}\) for \(i\ne j\); the diagonal is then set to zero.
This convention avoids the common but inconsistent simultaneous assertions
\(P=ZBZ^\top\) and \(P_{ii}=0\) when \(B_{aa}>0\).

Write \(\lambda_1(M)\ge\cdots\ge\lambda_n(M)\) for the algebraically ordered
eigenvalues of a symmetric \(M\).  Fix \(k<n\), and let \(U_\star\) and
\(\widehat U\) contain orthonormal bases for the top-\(k\) eigenspaces of
\(P\) and \(A\).  Only the boundary gap
\[
 g_k(P):=\lambda_k(P)-\lambda_{k+1}(P)
\tag{2.2}\label{eq:boundary-gap}
\]
is required; eigenvalue multiplicity inside the selected \(k\)-dimensional
cluster is harmless.  The parameter is the projector
\(\Pi_\star=U_\star U_\star^\top\).  For \(U,V\in\Gr(k,n)\), define
\[
 \dgr(U,V):=\norm{UU^\top-VV^\top}.
\tag{2.3}\label{eq:grassmann}
\]
This is the largest sine of the principal angles and is at most one.

\subsection{Matrix Bernstein with numerical constants}

Define the variance proxy
\[
 v(P):=\max_i\sum_{j\ne i}P_{ij}(1-P_{ij}).
\tag{2.4}\label{eq:vproxy}
\]
For \(x_{n,\alpha}:=\log(2n/\alpha)\), set
\[
 q_{n,\alpha}(\bar v)
 :=\sqrt{2\bar v\,x_{n,\alpha}}+\frac{2}{3}x_{n,\alpha}.
\tag{2.5}\label{eq:q}
\]

\begin{lemma}[Explicit adjacency quantile]\label{lem:bernstein}
Under \eqref{eq:model}, if \(v(P)\le\bar v\), then
\[
 \PP_P\!\left\{\norm{A-P}\le q_{n,\alpha}(\bar v)\right\}\ge1-\alpha.
\tag{2.6}\label{eq:bernstein}
\]
\end{lemma}

\begin{proof}
For \(i<j\), put
\[
 X_{ij}=(A_{ij}-P_{ij})(e_ie_j^\top+e_je_i^\top).
\]
Then \(A-P=\sum_{i<j}X_{ij}\), \(\E X_{ij}=0\), and
\(\norm{X_{ij}}\le1\).  Moreover,
\[
 \norm{\textstyle\sum_{i<j}\E X_{ij}^2}
 =\max_i\sum_{j\ne i}P_{ij}(1-P_{ij})=v(P).
\]
The two-sided self-adjoint matrix Bernstein inequality gives
\[
 \PP\{\norm{A-P}\ge t\}
 \le2n\exp\!\left\{-\frac{t^2}{2(v(P)+t/3)}\right\}.
\]
Substitution of \(t=\sqrt{2\bar v x}+2x/3\) makes the exponent at
least \(x=x_{n,\alpha}\), proving \eqref{eq:bernstein}
\citep{Tropp2012UserFriendlyTailBounds}.
\end{proof}

\begin{remark}\label{rem:dimension}
The factor \(2n\) is not cosmetic.  Replacing \(x_{n,\alpha}\) by
\(\log(2/\alpha)\) changes the numerical operating envelope.  Sharper
dimension-free random-graph bounds are available under additional expected
degree or regularisation conditions, but those conditions must then appear in
the certificate \citep{LeLevinaVershynin2017ConcentrationRegularization}.
\end{remark}

\section{Finite-sample eigenspace certificates}\label{sec:subspace}

\subsection{Model-envelope certificate}

\begin{theorem}[Grassmann confidence ball]\label{thm:grassmann}
Assume \(v(P)\le\bar v\) and \(g_k(P)\ge\underline g>0\), where
\(\bar v\) and \(\underline g\) are certified before observing \(A\).  Let
\[
 q=q_{n,\alpha}(\bar v),\qquad
 r_{\rm raw}=\frac{2q}{\underline g},\qquad
 r=\min\{1,r_{\rm raw}\}.
\tag{3.1}\label{eq:radius}
\]
Then
\[
 \mathcal C_{n,\alpha}
 =\{U\in\Gr(k,n):\dgr(U,\widehat U)\le r\}
\tag{3.2}\label{eq:ball}
\]
satisfies
\[
 \PP_P\{U_\star\in\mathcal C_{n,\alpha}\}\ge1-\alpha.
\tag{3.3}\label{eq:coverage}
\]
The certificate is informative precisely when \(r_{\rm raw}<1\).
\end{theorem}

\begin{proof}
On the event in Lemma~\ref{lem:bernstein}, the Davis--Kahan projector
inequality gives
\[
 \dgr(\widehat U,U_\star)
 \le\min\!\left\{1,\frac{2\norm{A-P}}{g_k(P)}\right\}
 \le r.
\]
The event has probability at least \(1-\alpha\).
\end{proof}

The theorem separates validity from informativeness.  If
\(r_{\rm raw}\ge1\), \(\mathcal C_{n,\alpha}\) is the whole
\(\Gr(k,n)\); its coverage is one but it carries no information.  Our
implementation reports both the coverage-guaranteed set and the binary field
\({\tt informative}=\mathbf1\{r_{\rm raw}<1\}\).

\subsection{A fully observable certificate}

The inputs \(\bar v\) and \(\underline g\) may be known in a designed SBM
experiment but are not known in an arbitrary observed network.  We therefore
give a conservative certificate computed from \(A\) alone.

Let \(D_i=\sum_{j\ne i}A_{ij}\).  For \(\alpha_d\in(0,1)\), put
\[
 \bar d(A,\alpha_d)
 =\min\!\left\{n-1,\,
 \max_i\left(\sqrt{\frac{x_d}{2}}+
 \sqrt{D_i+\frac{x_d}{2}}\right)^2\right\},
 \qquad x_d=\log(n/\alpha_d).
\tag{3.4}\label{eq:degree-ucb}
\]

\begin{lemma}[Expected-degree upper certificate]\label{lem:degree-ucb}
With probability at least \(1-\alpha_d\),
\[
 \max_i\sum_{j\ne i}P_{ij}\le\bar d(A,\alpha_d),
 \qquad v(P)\le\bar d(A,\alpha_d).
\tag{3.5}\label{eq:degree-ucb-event}
\]
\end{lemma}

\begin{proof}
Let \(\mu_i=\E D_i\).  The Chernoff lower-tail inequality
\(\PP\{\mu_i-D_i\ge\sqrt{2\mu_i x_d}\}\le e^{-x_d}\), followed by
the union bound, holds simultaneously for all \(i\).  Solving
\(\mu_i-D_i\le\sqrt{2\mu_i x_d}\) for \(\mu_i\) gives
\eqref{eq:degree-ucb}.  Finally,
\(\sum_jP_{ij}(1-P_{ij})\le\mu_i\).
\end{proof}

\begin{corollary}[Single-graph diagnostic]\label{cor:observable}
Choose \(\alpha_d+\alpha_q=\alpha\).  Compute
\[
 \bar d=\bar d(A,\alpha_d),\qquad
 q=q_{n,\alpha_q}(\bar d),\qquad
 \widehat g_k=\lambda_k(A)-\lambda_{k+1}(A),
\]
and
\[
 \underline g_{\rm obs}=\widehat g_k-2q.
\tag{3.6}\label{eq:observed-gap}
\]
If \(\underline g_{\rm obs}>0\), the ball centred at \(\widehat U\) with
radius
\[
 r_{\rm obs}
 =\min\!\left\{1,\frac{2q}{\underline g_{\rm obs}}\right\}
\tag{3.7}\label{eq:observed-radius}
\]
has coverage at least \(1-\alpha\).  It is nontrivial only if
\(\widehat g_k>4q\).  If either inequality fails, the protocol returns
``no gap certificate'' or ``valid but trivial,'' respectively.
\end{corollary}

\begin{proof}
On the event in Lemma~\ref{lem:degree-ucb},
\(q_{n,\alpha_q}(\bar d)\ge q_{n,\alpha_q}(v(P))\).  Hence
\[
 \PP\{\norm{A-P}>q,\ v(P)\le\bar d\}\le\alpha_q.
\]
Weyl's inequality gives \(g_k(P)\ge\widehat g_k-2q\) on
\(\{\norm{A-P}\le q\}\).  The union bound and
Theorem~\ref{thm:grassmann} finish the proof.
\end{proof}

\section{Downstream certificates and their computability}\label{sec:downstream}

\subsection{Approximate \texorpdfstring{\(k\)}{k}-means and a Hamming ball}

The clustering estimator must use \(\widehat U\), not the unknown population
alignment or population centres.  Let
\((\widehat\Theta,\widehat X)\) be a \(K\)-means solution accompanied by a
verified \((1+\varepsilon)\)-approximation guarantee:
\[
 \norm{\widehat\Theta\widehat X-\widehat U}_F^2
 \le(1+\varepsilon)
 \min_{\Theta,X}\norm{\Theta X-\widehat U}_F^2,
\tag{4.1}\label{eq:kmeans}
\]
where each row of the membership matrix \(\Theta\) has exactly one entry equal
to one and all other entries zero.
Let \(\widehat g\) denote the corresponding labels.

Assume \(U_\star=\Theta_\star X_\star\) has \(K\ge2\) distinct rows and
\[
 \Delta=\min_{a\ne b}\norm{(X_\star)_{a\cdot}-
 (X_\star)_{b\cdot}}_2>0.
\tag{4.2}\label{eq:margin}
\]
Labels are compared by permutation-invariant Hamming distance
\[
 d_H(g,h)=\min_{\pi\in\mathfrak S_K}
 \#\{i:g(i)\ne\pi(h(i))\}.
\]
Write
\[
 D_{K,n}=\left\lfloor n\left(1-\frac1K\right)\right\rfloor,\qquad
 \mathcal B_H(h,m)=\{g:d_H(g,h)\le m\}.
\tag{4.3}\label{eq:hamming-geometry}
\]
The averaging argument over label permutations gives
\(d_H(g,h)\le D_{K,n}\) for every pair \(g,h\).  Consequently a radius
\(m\ge D_{K,n}\), not merely a radius \(m\ge n\), covers the entire quotient
label space.  For \(K=2\), its diameter is exactly \(\lfloor n/2\rfloor\).

\begin{lemma}[Deterministic approximate-\(k\)-means transfer]\label{lem:kmeans}
Let \(Q^\circ\) minimise
\(\norm{\widehat UQ-U_\star}_F\) over \(Q\in O(k)\).  If each population
cluster retains at least one row outside the bad set used for centre matching,
then
\[
 d_H(\widehat g,g_\star)
 \le\frac{16+8\varepsilon}{\Delta^2}
 \norm{\widehat UQ^\circ-U_\star}_F^2.
 \tag{4.4}\label{eq:kmeans-bound}
\]
\end{lemma}

This is the explicit approximate-\(k\)-means lemma of
\citet{LeiRinaldo2015SpectralClusteringSBM}, written after orthogonal
alignment.  The cluster-retention condition is guaranteed whenever the
right-hand side of \eqref{eq:kmeans-bound} is below the smallest cluster size.
The estimator itself never uses \(Q^\circ\) or \(X_\star\); they appear only in
the proof.

For principal angles \(\theta_1,\ldots,\theta_k\),
\[
\min_Q\norm{\widehat UQ-U_\star}_F^2
=2\sum_{\ell=1}^k(1-\cos\theta_\ell)
\le2\sum_{\ell=1}^k\sin^2\theta_\ell
=\norm{\widehat U\widehat U^\top-U_\star U_\star^\top}_F^2
\le2k\,\dgr(\widehat U,U_\star)^2.
 \tag{4.5}\label{eq:principal-angles}
\]
This identity also corrects the invalid inequality
\(\norm{\widehat\Pi-\Pi_\star}_F^2\le
2\norm{\widehat\Pi-\Pi_\star}^2\) when \(k=2\).

\begin{theorem}[Model-envelope clustering confidence set]
\label{thm:clustering}
In addition to Theorem~\ref{thm:grassmann}, suppose that the population
envelopes
\[
 \Delta\ge\underline\Delta>0,\qquad
 n_{\min}:=\min_a\#\{i:g_\star(i)=a\}\ge\underline n_{\min}
\]
are certified before observing \(A\), and that the computed \(K\)-means
solution is accompanied by \eqref{eq:kmeans} with a verified
\(\varepsilon\le\bar\varepsilon\).  Define
\[
 M_{n,\alpha}
 =\frac{(32+16\bar\varepsilon)k\,r^2}{\underline\Delta^2},
 \qquad m_{n,\alpha}=\lfloor M_{n,\alpha}\rfloor.
\tag{4.6}\label{eq:hamming-radius}
\]
If \(M_{n,\alpha}<\underline n_{\min}\), then
\[
 \PP\{g_\star\in\mathcal B_H(\widehat g,m_{n,\alpha})\}
 \ge1-\alpha.
\tag{4.7}\label{eq:hamming-coverage}
\]
Moreover \(m_{n,\alpha}<D_{K,n}\), so the Hamming ball is proper.  If the
retention inequality or the algorithmic approximation guarantee is not
certified, the protocol returns ``no clustering certificate.''
\end{theorem}

\begin{proof}
On the event of Theorem~\ref{thm:grassmann},
Lemma~\ref{lem:kmeans} and \eqref{eq:principal-angles} bound both the bad-row
count and \(d_H(\widehat g,g_\star)\) by \(M_{n,\alpha}\).  The strict
inequality \(M_{n,\alpha}<\underline n_{\min}\le n_{\min}\) guarantees that
every population cluster retains a good row, so the centre-matching
permutation exists.  Since Hamming distance is integer,
\(d_H\le\lfloor M_{n,\alpha}\rfloor\).  Finally,
\(\underline n_{\min}\le n_{\min}\le\lfloor n/K\rfloor\le D_{K,n}\),
proving properness.
\end{proof}

The distinction between a computable label vector and a computable
\emph{certificate} matters.  Lloyd's heuristic, including the implementation
used descriptively in Section~\ref{sec:simulation}, does not ordinarily return
a verified approximation factor.  We therefore do not attach the theorem's
confidence label to those empirical partitions.  Exact recovery also requires
a genuine rowwise bound; the \(\ell_p\) theory of
\citet{AbbeFanWang2022LpPCASpectralClustering} provides such results under
additional model-specific assumptions.  Those assumptions are not silently
imported into the present general model.

\subsection{Degree centrality with an observed selection margin}

Let \(d_i(P)=\sum_{j\ne i}P_{ij}\) and \(D_i=d_i(A)\).
If
\[
 v_{\deg}:=\max_i\sum_{j\ne i}P_{ij}(1-P_{ij})\le\bar v_{\deg},
\]
define
\[
 h_{n,\alpha}(\bar v_{\deg})
 =\sqrt{2\bar v_{\deg}\log(2n/\alpha)}
 +\frac23\log(2n/\alpha).
 \tag{4.8}\label{eq:degree-half}
\]

\begin{theorem}[Simultaneous degree bands and top-\(m\) certificate]
\label{thm:degree}
With probability at least \(1-\alpha\),
\[
 d_i(P)\in[D_i-h_{n,\alpha},D_i+h_{n,\alpha}]
 \quad\text{for every }i.
 \tag{4.9}\label{eq:degree-bands}
\]
Let \(\widehat S_m\) be the \(m\) largest observed degrees and
\[
 \widehat\Gamma_m
 =\min_{i\in\widehat S_m}D_i-\max_{j\notin\widehat S_m}D_j.
 \tag{4.10}\label{eq:observed-margin}
\]
If \(\widehat\Gamma_m>2h_{n,\alpha}\), then, on the same event,
\(\widehat S_m\) is the unique population top-\(m\) degree set.
\end{theorem}

\begin{proof}
Scalar Bernstein and a union bound give \eqref{eq:degree-bands}.  For
\(i\in\widehat S_m\), \(j\notin\widehat S_m\),
\[
 d_i(P)-d_j(P)\ge D_i-D_j-2h_{n,\alpha}>0.
\]
\end{proof}

The condition uses an observed margin.  It therefore produces an actual
decision rule: certify if \(\widehat\Gamma_m>2h\), otherwise decline.
For a fully observable version, choose
\(\alpha_d+\alpha_s=\alpha\) and substitute
\[
 \bar v_{\deg}(A)
 =\min\{\bar d(A,\alpha_d),(n-1)/4\}
\]
into \(h_{n,\alpha_s}\).  Lemma~\ref{lem:degree-ucb}, the deterministic
Bernoulli-variance bound, and a union bound give simultaneous degree-band and
selection coverage at least \(1-\alpha\).

\subsection{A corrected normalized-Katz extension}

For \(s=\1/\sqrt n\), define normalized Katz centrality
\[
 c_\beta(M)=(I-\beta M)^{-1}s-s.
 \tag{4.11}\label{eq:katz}
\]
Normalising the seed is essential when operator norm controls the
perturbation.  Without it, the Lipschitz modulus below contains the previously
omitted factor \(\sqrt n\).

\begin{proposition}[Katz band]\label{prop:katz}
Suppose \(\beta>0\), \(\norm{M},\norm{M'}\le R\), and
\(\beta R\le\tau<1\).  Then
\[
 \norm{c_\beta(M)-c_\beta(M')}_\infty
 \le\frac{\beta}{(1-\tau)^2}\norm{M-M'}.
 \tag{4.12}\label{eq:katz-lip}
\]
Consequently, if \(\norm{P}\le R_0\), \(q\) is a valid operator quantile and
\(\beta(R_0+q)\le\tau\), then simultaneous bands of half-width
\[
 h_K=\frac{\beta q}{(1-\tau)^2}
 \tag{4.13}\label{eq:katz-half}
\]
hold with the same coverage.  An observed top-\(m\) Katz set is certified
when its observed boundary gap exceeds \(2h_K\).
\end{proposition}

\begin{proof}
The resolvent identity gives
\[
 c_\beta(M)-c_\beta(M')
 =\beta(I-\beta M)^{-1}(M-M')
 (I-\beta M')^{-1}s.
\]
Take the Euclidean norm, use \(\norm{s}_2=1\) and then
\(\norm{x}_\infty\le\norm{x}_2\).  On
\(\{\norm{A-P}\le q\}\), both \(A\) and \(P\) have norm at most
\(R_0+q\).
\end{proof}

The resolvent domain can also be checked from the observed graph: if a valid
deterministic operator quantile \(q\) is available and
\(\beta(\norm{A}+q)\le\tau<1\), then on
\(\{\norm{A-P}\le q\}\) both \(\norm A\) and \(\norm P\) are at most
\(\norm A+q\), and the same half-width applies.  A data-dependent \(q\) from
Corollary~\ref{cor:observable} requires its stated split error budget.

\section{Operating envelopes and corrected examples}\label{sec:envelope}

\subsection{Balanced two-block SBM}

Let \(n=2m\), with two equal blocks, within-block probability \(p\) and
between-block probability \(q<p\).  With the diagonal set to zero,
\[
 \lambda_1(P)=(m-1)p+mq,\qquad
 \lambda_2(P)=(m-1)p-mq,\qquad
 \lambda_3(P)=\cdots=\lambda_n(P)=-p.
\tag{5.1}\label{eq:sbm-spectrum}
\]
Thus the top-two boundary gap, maximum expected degree and exact variance
proxy are
\[
\begin{split}
 g_2(P)&=m(p-q),\\
 d_{\max}&=(m-1)p+mq,\\
 v(P)&=(m-1)p(1-p)+mq(1-q).
\end{split}
\tag{5.2}\label{eq:sbm-inputs}
\]
Notice that \(\lambda_1-\lambda_2\) is an internal gap and is irrelevant to
the two-dimensional subspace.

For the previously used values \(n=200,p=0.30,q=0.10,\alpha=0.05\),
\[
 v(P)=29.79,\quad d_{\max}=39.7,\quad g_2(P)=20.
\]
Equation~\eqref{eq:q} gives
\[
 q_{200,0.05}(v)=29.131,\qquad
 r_{\rm raw}=2.913.
\tag{5.3}\label{eq:old-corrected}
\]
Using \(d_{\max}\) instead gives \(r_{\rm raw}=3.270\).  Both exceed one,
so the confidence region is the entire Grassmannian.  The formula with the
dimension factor omitted and an unspecified common constant \(C=1\) would
have reported \(1.579\), which is already vacuous and is not a valid numerical
matrix-Bernstein calibration.

\begin{figure}[t]
\centering
\includegraphics[width=0.72\textwidth]{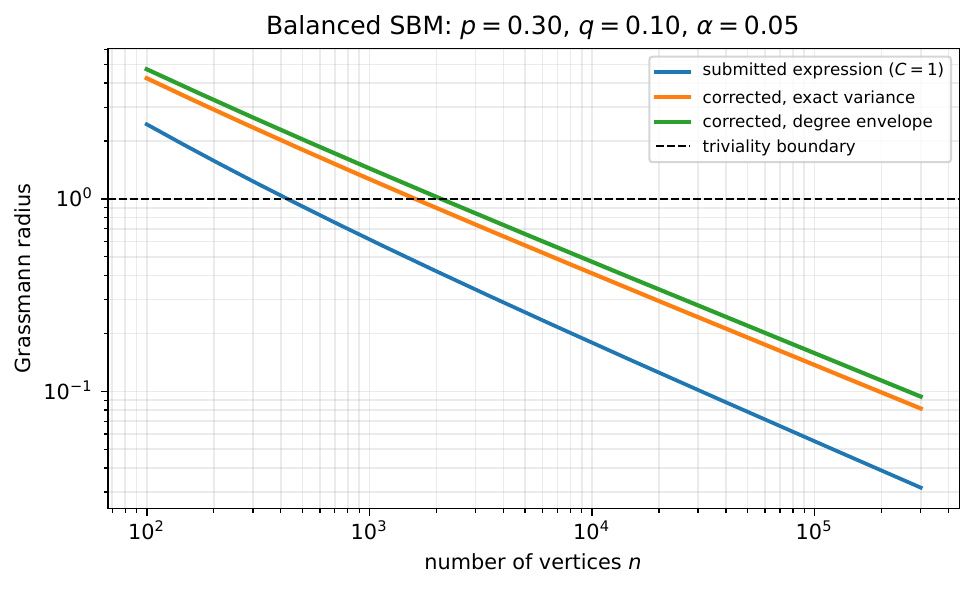}
\caption{Submitted and corrected radii for the balanced SBM.  The corrected
curves retain the dimension factor and numerical constants.  Radius one is
the diameter of the Grassmannian in projector norm.}
\label{fig:radius}
\end{figure}

Figure~\ref{fig:radius} follows the same \(p,q,\alpha\) while varying \(n\).
With the exact variance, the first even \(n\) with radius below one is
\(n=\ExactRadiusThreshold\); the degree-envelope version first crosses at
\(n=\DegreeRadiusThreshold\).  The population eigenbasis has rows
\((1/\sqrt n,\pm1/\sqrt n)\), so \(\Delta^2=4/n\).  For exact \(k\)-means
(\(\varepsilon=0\)), \eqref{eq:hamming-radius} reduces to
\[
 M_{n,\alpha}=16nr^2,\qquad m_{n,\alpha}=\lfloor16nr^2\rfloor.
\tag{5.4}\label{eq:balanced-hamming}
\]
Because the smallest cluster size and quotient-label-space diameter are both
\(n/2\), the retention and proper-ball condition is
\(16nr^2<n/2\).  It first holds at \(n=\HammingThreshold\), found by an exact search
over even integers.  The former comparison with \(n\) would be wrong after
label permutations are identified.  This slow transition is the structural
limitation of mean-square propagation; it is not hidden by asymptotic notation.

\begin{figure}[t]
\centering
\includegraphics[width=0.76\textwidth]{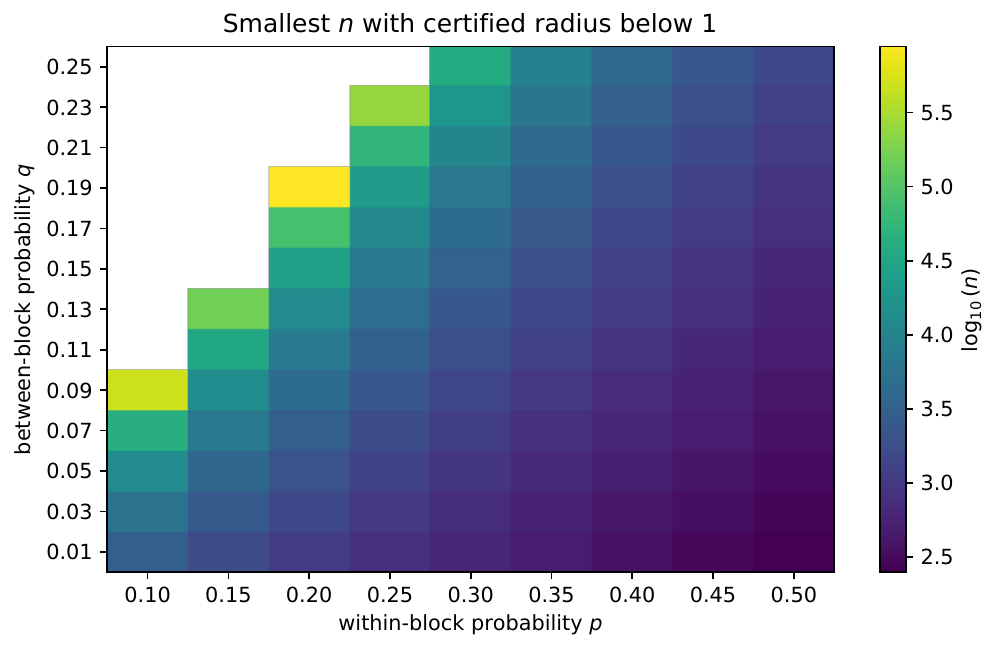}
\caption{Exact smallest even \(n\) for which the exact-variance Grassmann
radius is below one over a grid of balanced-SBM probabilities.  Blank cells
have \(q\ge p\) or do not cross by \(n=1{,}000{,}000\).}
\label{fig:regime}
\end{figure}

Figure~\ref{fig:regime} maps the operating envelope over \(p\) and \(q\).
Strong within/between separation produces a usable subspace certificate at
moderate \(n\), whereas \(p\approx q\) pushes the crossing far to the right.

\subsection{Centrality examples with genuine margins}

Consider \(n=2000\), two equal blocks and
\[
 B=\begin{pmatrix}0.60&0.02\\0.02&0.10\end{pmatrix}.
\tag{5.5}\label{eq:unequal-B}
\]
The population expected degrees are \(619.4\) and \(119.9\), giving margin
\(\UnequalDegreeMargin\).  The simultaneous Bernstein half-width is
\(\UnequalDegreeHalf\).  In the fixed reproducible draw, the observed
boundary gap between the first \(1000\) and remaining vertices is
\(\UnequalObservedGap>2(\UnequalDegreeHalf)\); Theorem~\ref{thm:degree}
therefore certifies the selected high-centrality block.  The top set agrees
with the first block for every vertex.

The same model also corrects the zero-margin Katz example.  With
\(\beta=1/(4\rho(P))\), the two unnormalised population Katz scores are
\(0.34446\) and \(0.05510\), so the population margin is \(0.28935\), not
zero.  The general global operator band remains too conservative to certify
this moderately noisy model, and the protocol says so.  A separate
near-deterministic operating example with \(n=5000\),
\(p_{11}=0.999999,p_{22}=10^{-6},q=10^{-7}\) has normalized-Katz population
margin more than \(\KatzMarginRatio\) times \(2h_K\); in that deliberately
extreme regime the observed-gap certificate fires on the coverage event.
Reporting both cases exposes the narrow operating envelope of the global Katz
bound.

\subsection{Degree-ranking coverage and certificate firing}

To validate the centrality rule separately from the eigenspace study, we
generated \(1000\) graphs from the unequal two-block model
\eqref{eq:unequal-B} at \(n=500\), with selection size \(250\).  The variance
envelope and selection size were fixed before simulation.  Table
\ref{tab:centrality-sim} reports simultaneous band coverage, the frequency
with which the observed-gap certificate fired, and exact recovery of the
population top set.  Every fired certificate was correct
(\(\CentralityCorrectGivenFire\)); the firing rate was
\(\CentralityFiring\), so the design produces both ``certificate'' and
``no certificate'' outcomes rather than only a favourable case.

\begin{table}[t]
\centering
\caption{Degree-centrality validation at \(n=500\) over \(1000\)
replications.  Intervals are exact two-sided \(95\%\) Clopper--Pearson
intervals for Monte Carlo proportions.}
\label{tab:centrality-sim}
\begin{tabular}{lrrr}
\toprule
Diagnostic&Count&Rate&Monte Carlo interval\\
\midrule
Simultaneous degree band&1000/1000&1.000&[0.996, 1.000]\\
Observed-gap certificate fired&711/1000&0.711&[0.682, 0.739]\\
Selected set exactly correct&1000/1000&1.000&[0.996, 1.000]\\
\bottomrule
\end{tabular}
\end{table}

\section{Simulation and bootstrap comparison}\label{sec:simulation}

\subsection{Design}

We considered \(n\in\{200,500,1000,2000\}\) and three balanced-SBM pairs:
\[
(p,q)\in\{(0.30,0.10),(0.20,0.05),(0.30,0.20)\}.
\]
For each of the \(12\) cells we generated \(B=1000\) graphs using seed
\(2026072901\).  We computed the exact-variance radius, realised
\(\dgr(\widehat U,U_\star)\), its empirical \(0.95\) quantile, approximate
\(k\)-means labels and permutation-invariant Hamming error.  The radius is
reported as \(\min(1,r_{\rm raw})\), while the raw value and both
operating flags remain in the supplied CSV file.  The \(H\) flag is the
model-envelope condition for an ideal globally optimal \(k\)-means solution;
it is not attached to the descriptive Lloyd output.

\begin{figure}[t]
\centering
\includegraphics[width=\textwidth]{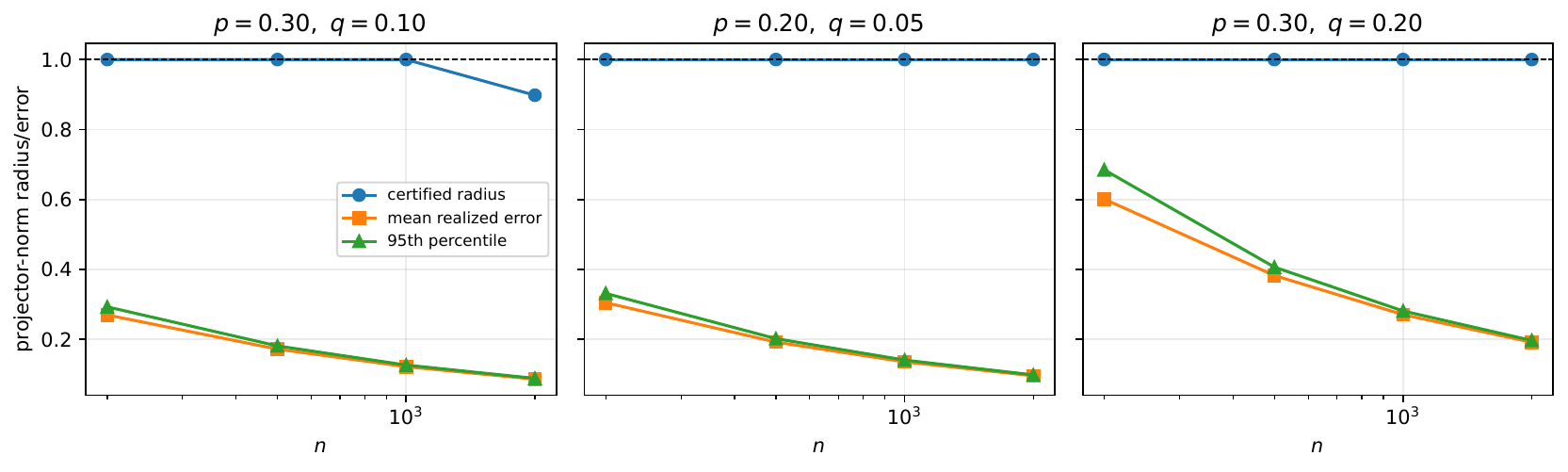}
\caption{Certified radius, mean realised Grassmann error and empirical
95th percentile across \(1000\) replications per point.  A certified radius
equal to one has trivial coverage.}
\label{fig:coverage}
\end{figure}

\begin{table}[t]
\centering
\caption{Finite-sample operating envelope.  Each row uses \(1000\)
replications.  \(S\) indicates a nontrivial subspace ball and \(H\) the
availability of the model-envelope Hamming certificate for ideal exact
\(k\)-means; ``cov.'' is empirical subspace coverage, \(\bar e\) the mean
realised Grassmann error, \(e_{.95}\) its 95th percentile, and ``exact'' the
fraction of zero-Hamming descriptive Lloyd outputs.  The reported radius is
\(\min(1,r_{\rm raw})\).}
\label{tab:simulation}
\footnotesize
\setlength{\tabcolsep}{3.5pt}
\begin{tabular}{rrrrccrrrr}
\toprule
\(n\)&\(p\)&\(q\)&\(r_{\rm raw}\)&\(S\)&\(H\)&cov.&
\(\bar e\)&\(e_{.95}\)&exact\\
\midrule
200&0.30&0.10&2.913&no&no&1.000&0.270&0.293&0.959\\
200&0.20&0.05&3.364&no&no&1.000&0.305&0.332&0.860\\
200&0.30&0.20&6.341&no&no&1.000&0.601&0.685&0.000\\
500&0.30&0.10&1.804&no&no&1.000&0.172&0.182&1.000\\
500&0.20&0.05&2.059&no&no&1.000&0.192&0.202&1.000\\
500&0.30&0.20&3.949&no&no&1.000&0.383&0.407&0.014\\
1000&0.30&0.10&1.268&no&no&1.000&0.122&0.127&1.000\\
1000&0.20&0.05&1.438&no&no&1.000&0.136&0.141&1.000\\
1000&0.30&0.20&2.786&no&no&1.000&0.271&0.281&0.810\\
2000&0.30&0.10&0.898&yes&no&1.000&0.086&0.089&1.000\\
2000&0.20&0.05&1.013&no&no&1.000&0.096&0.099&1.000\\
2000&0.30&0.20&1.978&no&no&1.000&0.192&0.197&1.000\\
\bottomrule
\end{tabular}
\end{table}

Figure~\ref{fig:coverage} and Table~\ref{tab:simulation} show the central
empirical fact.  The realised error is much smaller than the uniform analytic
radius.  At \(n=200,p=0.30,q=0.10\), for example, the mean error is about
\(0.27\), in agreement with the independent editorial reconstruction, while
the corrected raw radius is \(2.913\).  Thus empirical algorithmic success
does not make the confidence ball informative.  Only one of the \(12\) cells
(\(8.3\%\)) has \(S=\mathrm{yes}\), and none has \(H=\mathrm{yes}\).
All subspace balls use the same projector metric on \(\Gr(2,n)\); at fixed
\(n\), their metric-ball volume is monotone in radius.  We therefore report
radius as the transparent size index rather than claiming a closed-form
geodesic volume.

\subsection{Bootstrap benchmark}

At \(n=200,p=0.30,q=0.10\) we additionally used \(100\) outer graphs and
\(100\) bootstrap graphs within each outer replication.  We compared:
\begin{enumerate}[leftmargin=18pt,label=(\roman*)]
\item the analytic certificate;
\item a plug-in two-block SBM bootstrap fitted after spectral \(k\)-means; and
\item a simplified ASE-\(k\)-nearest-neighbour probability smoother with
\(k=20\), based on the construction studied by
\citet{DilworthDavisLawson2025ValidBootstrapsNetworkEmbeddings}.
\end{enumerate}
For each bootstrap, the radius is the conditional \(0.95\) quantile of the
Grassmann displacement from the observed embedding.

\begin{table}[t]
\centering
\caption{Bootstrap comparison at \(n=200,p=0.30,q=0.10\).
Bootstrap radii are not presented as finite-sample certificates; the table
measures their empirical coverage and size.}
\label{tab:bootstrap}
\begin{tabular}{lrrr}
\toprule
Method&Empirical coverage&Mean radius&Median radius\\
\midrule
Analytic certificate&1.000&1.000&1.000\\
Plug-in SBM bootstrap&1.000&0.402&0.402\\
Simplified ASE-$k$NN bootstrap&0.870&0.288&0.287\\
\bottomrule
\end{tabular}
\end{table}

The comparison has a specific interpretation.  The analytic method protects
coverage by returning the whole Grassmannian in this cell.  The bootstrap
methods return substantially smaller regions, but their coverage depends on
the probability-matrix estimator.  The plug-in SBM bootstrap covered all
\(100\) outer targets with mean radius \(0.402\), whereas the smaller
simplified ASE-\(k\)NN radius averaged \(0.288\) and covered only \(0.870\).
We did not implement that paper's exchangeability validation step and make no
claim that this simplified comparator is its full procedure.  This is not
a contest in which the smallest radius automatically wins: radius and
coverage must be judged jointly, and neither bootstrap number is promoted to
an exact finite-sample guarantee here.

\subsection{Monte Carlo precision}

The simulation counts also quantify their own numerical uncertainty.  In
each primary cell the analytic ball covered all \(1000\) realised targets;
the two-sided \(95\%\) Clopper--Pearson interval for a \(1000/1000\) count is
\([0.9963,1]\).  In the bootstrap experiment, \(100/100\) plug-in-SBM
coverages correspond to \([0.9638,1]\), while \(87/100\) simplified
ASE-\(k\)NN
coverages correspond to \([0.7880,0.9289]\).  These intervals describe Monte
Carlo error only.  They do not convert an empirical bootstrap comparison
into a theorem, but they show that the observed simplified ASE-\(k\)NN
shortfall is
larger than simulation noise at the nominal \(0.95\) level.

\section{Real-network diagnostic}\label{sec:real}

We applied the fully observable protocol to the unweighted Zachary
karate-club network, which has \(34\) vertices, \(78\) edges and a recorded
two-faction split \citep{Zachary1977KarateClub}.  Spectral \(k\)-means makes
\(\KarateHamming\) error after label permutation, giving agreement
\(\KarateAccuracy\).  This descriptive accuracy is high.

The inferential diagnostic is different.  Both diagnostics use a two-part
error budget: the expected-degree envelope receives \(0.025\), and the
corresponding matrix or scalar tail receives \(0.025\).  The boundary gap
is \(\KarateBoundaryGap\), while the split-level operator quantile is
\(\KarateOperatorQuantile\).  Consequently the Weyl lower bound
\(\widehat g_2-2q=\KarateGapLower\) is nonpositive and
Corollary~\ref{cor:observable} returns ``no gap certificate.''  For degree
centrality, the largest and second-largest observed degrees are \(17\) and
\(16\); the observed margin \(1\) is below twice the simultaneous half-width
\(\KarateDegreeHalf\), so the top-one degree selection is also not certified.

\begin{figure}[t]
\centering
\includegraphics[width=0.58\textwidth]{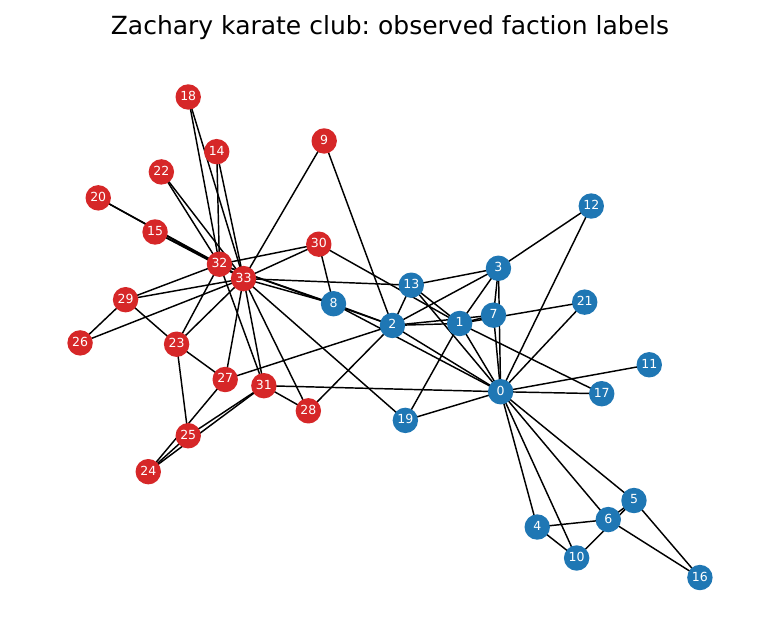}
\caption{Zachary karate-club network with the recorded post-fission faction
labels.  The spectral partition has high descriptive agreement, but the
finite-sample protocol correctly returns no nontrivial certificate.}
\label{fig:karate}
\end{figure}

This example demonstrates why ``no certificate'' is a useful output.
Refusing to certify a small network does not assert that the partition is
wrong.  It says that this distribution-free envelope, from this single
network, cannot support a nontrivial uniform confidence claim.  A parametric
SBM analysis or a validated network bootstrap could be more informative, but
would answer a question conditional on additional assumptions.

\section{Why a gap condition cannot be removed}\label{sec:lower}

The earlier claim based on two distinct \(P_0,P_1\) inducing exactly the same
Bernoulli graph law is impossible: under \eqref{eq:model},
\(P_{ij}=\E_P A_{ij}\), so equality of laws implies \(P_0=P_1\).  The correct
finite-sample obstruction uses nearby, not identical, laws.

\begin{theorem}[Two-point diameter lower bound]\label{thm:lecam}
Let \(P_0,P_1\) induce graph laws \(\mathbb P_0,\mathbb P_1\) and targets
\(U_0,U_1\in\Gr(k,n)\) with \(\dgr(U_0,U_1)\ge\delta\).  If a random set
\(\widehat{\mathcal C}(A)\) satisfies
\[
\mathbb P_j\{U_j\in\widehat{\mathcal C}(A)\}\ge1-\alpha,
\qquad j=0,1,
\]
then
\[
\mathbb P_0\{\diam(\widehat{\mathcal C})\ge\delta\}
\ge1-2\alpha-\TV(\mathbb P_0,\mathbb P_1).
\tag{8.1}\label{eq:lecam}
\]
\end{theorem}

\begin{proof}
Total variation implies
\(\mathbb P_0\{U_1\in\widehat{\mathcal C}\}
\ge1-\alpha-\TV(\mathbb P_0,\mathbb P_1)\).
Intersect this event with
\(\{U_0\in\widehat{\mathcal C}\}\) and use the union bound.
\end{proof}

The following explicit collision sequence embeds the usual Le Cam argument
inside model \eqref{eq:model}.  Let \(n\) admit four mutually orthogonal
Rademacher vectors \(u,v,w,z\), each orthogonal to \(\1\), with entries
\(\pm n^{-1/2}\).  Set \(\bar P=p(\1\1^\top-I)\), \(0<p<1\), and
\[
 P_{0,\tau}=\bar P+\tau(uu^\top-vv^\top),\qquad
 P_{1,\tau}=\bar P+\tau(ww^\top-zz^\top).
\tag{8.2}\label{eq:collision}
\]
Both perturbations have zero diagonal because the squared entries of the
Rademacher vectors agree.  For sufficiently small \(\tau\), all off-diagonal
entries remain in \((0,1)\).  The top-two subspaces are
\(\operatorname{span}\{\1,u\}\) and
\(\operatorname{span}\{\1,w\}\), whose Grassmann distance is one, while their
boundary gaps are of order \(\tau\).

If the off-diagonal probabilities are bounded in
\([\eta,1-\eta]\), Bernoulli product-law calculus gives
\[
 \KL(\mathbb P_{0,\tau},\mathbb P_{1,\tau})
 \le C_\eta\norm{P_{0,\tau}-P_{1,\tau}}_F^2
 =O(\tau^2).
\]
Pinsker's inequality then yields
\(\TV(\mathbb P_{0,\tau},\mathbb P_{1,\tau})=O(\tau)\)
\citep{Tsybakov2009NonparametricEstimation}.  Theorem~\ref{thm:lecam}
forces diameter close to one with probability at least \(1-2\alpha-o(1)\)
as the gap vanishes.  This is a genuine lower bound, not a restatement of
nonuniqueness at exact collision.

\section{Discussion}\label{sec:discussion}

The analysis yields three distinctions that should be retained in applied
network work.

First, validity and informativeness are different.  A set equal to the entire
parameter space can have perfect coverage.  Reporting its nominal level
without its radius conceals the essential result.  Our output therefore
includes the raw radius, the diameter-capped radius, a gap-certificate flag
and a downstream-informativeness flag.

Second, algorithmic accuracy and inferential certification are different.
Spectral clustering may recover nearly all labels in the simulation and real
network while the uniform analytic region is trivial.  The certificate is
designed to be sufficient, not necessary.  Its conservatism is visible in
Figure~\ref{fig:coverage}.

Third, downstream discontinuities require observed margins.  Top-\(m\)
selection cannot be certified at a tie.  Conditions stated only in terms of an
unknown population margin are mathematically meaningful but operationally
incomplete.  The observed-gap rules in Theorem~\ref{thm:degree} and
Proposition~\ref{prop:katz} turn the margin into a computable decision.
Clustering is different: the present Frobenius route still needs declared
population separation and minimum-cluster envelopes, plus a verified
algorithmic approximation factor.  We label it a model-envelope certificate
rather than calling it fully observable.

\subsection{Scope and limitations}

The guarantees are conditional on a fixed probability matrix and independent
Bernoulli edges.  They do not automatically cover temporal dependence,
degree-preserving sampling, missing edges or a network selected after looking
at the same data.  Such mechanisms require a concentration inequality matched
to the actual sampling design.

The analytic balls are uniform operator-norm certificates, not claims of
optimal diameter.  They deliberately ignore entrywise structure that can
make a particular SBM or random-dot-product graph much easier.  This explains
why the karate partition can be descriptively accurate while the certificate
is trivial.  A model-specific bootstrap or rowwise theorem may be sharper,
but its extra assumptions must be reported as part of the inferential target.
In particular, ordinary Lloyd \(k\)-means does not by itself verify the
approximation factor required in Theorem~\ref{thm:clustering}; its partitions
in our simulations and application are descriptive outputs.

Finally, the method certifies a prespecified dimension \(k\), centrality
parameter \(\beta\), and selection size \(m\).  Choosing these quantities from
the observed graph and then applying the same nominal error budget entails
selection effects not covered by the stated results.  Sample splitting,
simultaneous calibration over a finite candidate set or a separately
validated model-selection procedure would be needed for that extension.

Several extensions are natural.  Regularised adjacency or Laplacian operators
can reduce sparse-graph concentration error
\citep{LeLevinaVershynin2017ConcentrationRegularization}; their finite-sample
certificates require the corresponding operator-specific bias and gap
analysis.  Rowwise eigenvector theory can produce much sharper clustering
statements than the mean-square Hamming ball
\citep{AbbeFanWang2022LpPCASpectralClustering}.  Finally, validated network
bootstraps can trade analytic conservatism for estimator-dependent
calibration.  These are complementary approaches, not interchangeable
justifications.

\section{Conclusion}\label{sec:conclusion}

We have constructed a finite-sample protocol for spectral graph procedures
that is explicit about when it has no useful answer.  Every constant in the
matrix-Bernstein radius is numerical, the ambient dimension appears in the
quantile, and the eigenspace target is orthogonally invariant.  The clustering
result now states all population-envelope, cluster-retention and algorithmic
conditions needed for a proper quotient-space Hamming ball.  Centrality
selection uses observed margins.  Simulations and a real-network application
show why these details matter: conservative coverage can coexist with a
vacuous set, and an accurate-looking output need not be certifiable.

The practical recommendation is simple.  Report the certificate only together
with its operating diagnostics.  If the gap or margin test fails, report
``no nontrivial certificate'' rather than an apparently precise but
unsupported number.

\appendix

\section{Proof details for the clustering transfer}\label{app:kmeans}

We include the deterministic argument behind Lemma~\ref{lem:kmeans} to make
the constants and the computability distinction auditable.  Rotate
\(\widehat U\) by \(Q^\circ\) and write
\(\widetilde U=\widehat UQ^\circ\).  Rotation does not change the
\(k\)-means objective or its assignments.  Put
\[
 E=\norm{\widetilde U-U_\star}_F,\qquad
 \overline U=\widehat\Theta\widehat XQ^\circ.
\]
Because \(U_\star\) itself is a feasible \(K\)-centre matrix for the
\(k\)-means problem,
\[
 \norm{\overline U-\widetilde U}_F
 \le\sqrt{1+\varepsilon}\,
 \norm{U_\star-\widetilde U}_F.
\tag{A.1}\label{eq:app-opt}
\]
The triangle inequality and
\[
 \bigl(1+\sqrt{1+\varepsilon}\bigr)^2
 =2+\varepsilon+2\sqrt{1+\varepsilon}
 \le4+2\varepsilon
\]
therefore give
\[
 \norm{\overline U-U_\star}_F^2
 \le(4+2\varepsilon)E^2.
\tag{A.2}\label{eq:app-centre}
\]

For population cluster \(a\), define
\[
 S_a=\left\{i:g_\star(i)=a,\
 \norm{\overline u_i-u_{\star i}}_2\ge\frac{\Delta}{2}\right\}.
\tag{A.3}\label{eq:bad}
\]
Then
\[
 \frac{\Delta^2}{4}\sum_{a=1}^K|S_a|
 \le\norm{\overline U-U_\star}_F^2
 \le(4+2\varepsilon)E^2,
\]
and hence
\[
 \sum_{a=1}^K|S_a|
 \le\frac{16+8\varepsilon}{\Delta^2}E^2.
\tag{A.4}\label{eq:bad-count}
\]
If every cluster contains a vertex outside \(S_a\), choose one such vertex.
Its fitted centre lies strictly within \(\Delta/2\) of the corresponding
population centre.  Two different population clusters cannot map to the same
fitted centre, since that would put their centres at distance less than
\(\Delta\).  The resulting map of centres is injective and, because there are
\(K\) of each, is a permutation.  Every vertex outside
\(\cup_aS_a\) is then correctly matched under that permutation, proving
\eqref{eq:kmeans-bound}.

For completeness, if \(\theta_1,\ldots,\theta_k\) are the principal angles,
the Procrustes and projector formulae are
\[
\begin{split}
\min_Q\norm{\widehat UQ-U_\star}_F^2
 &=2\sum_{\ell=1}^k(1-\cos\theta_\ell),\\
\norm{\widehat U\widehat U^\top-U_\star U_\star^\top}_F^2
 &=2\sum_{\ell=1}^k\sin^2\theta_\ell,\\
\norm{\widehat U\widehat U^\top-U_\star U_\star^\top}^2
 &=\max_\ell\sin^2\theta_\ell.
\end{split}
\tag{A.5}\label{eq:app-angles}
\]
Since \(1-\cos\theta\le\sin^2\theta\), the chain in
\eqref{eq:principal-angles} follows.  This establishes the constant in
Theorem~\ref{thm:clustering} without using a population quantity in the
algorithm.

\section{Exact two-block calculations}\label{app:block}

\subsection{Balanced spectrum and margin}

For a balanced two-block SBM, decompose \(\R^n\) into the span of
\(\1\), the signed block vector \(s=(\1_m,-\1_m)\), and the two
within-block zero-sum subspaces.  Direct multiplication gives
\[
\begin{split}
P\1&=((m-1)p+mq)\1,\\
Ps&=((m-1)p-mq)s,\\
Pv&=-pv
\end{split}
\]
for every within-block zero-sum \(v\).  This proves
\eqref{eq:sbm-spectrum}.  The top-two eigenbasis can be chosen as
\[
 U_\star=\begin{pmatrix}\1/\sqrt n&s/\sqrt n\end{pmatrix}.
\]
The two row types are \((1/\sqrt n,\pm1/\sqrt n)\), so their separation is
\(\Delta=2/\sqrt n\).  Substituting \(k=2,\bar\varepsilon=0\) and
\(\Delta^2=4/n\) into \eqref{eq:hamming-radius} yields
\eqref{eq:balanced-hamming}.

\subsection{Unequal-block quotient and Katz scores}

Let block sizes be \(n_1,n_2\) and
\[
 B=\begin{pmatrix}p_{11}&q\\q&p_{22}\end{pmatrix}.
\]
The block-constant invariant subspace is represented by the quotient
\[
 Q_B=
 \begin{pmatrix}
 (n_1-1)p_{11}&n_2q\\
 n_1q&(n_2-1)p_{22}
 \end{pmatrix}.
\tag{B.1}\label{eq:quotient}
\]
The remaining eigenvalues are \(-p_{11}\) with multiplicity \(n_1-1\)
and \(-p_{22}\) with multiplicity \(n_2-1\).  This provides all spectral
inputs without forming an \(n\times n\) population matrix.

For the unnormalised Katz seed \(\1\), the two population score values
\((c_1,c_2)\) satisfy
\[
 \begin{pmatrix}c_1+1\\c_2+1\end{pmatrix}
 =(I_2-\beta Q_B)^{-1}\begin{pmatrix}1\\1\end{pmatrix}.
\tag{B.2}\label{eq:katz-quotient}
\]
In the moderate model \eqref{eq:unequal-B}, this gives the values stated in
Section~\ref{sec:envelope}.  Unlike the rejected equal-row-sum example,
\(Q_B\1_2\) is not proportional to \(\1_2\), and the score margin is
strictly positive.

\section{Details for the near-collision lower bound}
\label{app:collision}

Here we verify the spectral and distributional claims following
\eqref{eq:collision}.  Because \(\1^\top u=\1^\top v=0\), the action of
\(\bar P=p(\1\1^\top-I)\) is multiplication by \(p(n-1)\) on
\(\operatorname{span}\{\1\}\) and by \(-p\) on its orthogonal complement.
Consequently, \(P_{0,\tau}\) has eigenvalues
\[
 p(n-1),\quad -p+\tau,\quad -p-\tau,\quad
 \underbrace{-p,\ldots,-p}_{n-3},
\]
with eigenvectors \(\1,u,v\) and the remaining orthogonal directions.
The analogous distinguished directions for \(P_{1,\tau}\) are \(w,z\).
Thus the top-two boundary gap is exactly \(\tau\), and the two top-two
subspaces share only \(\operatorname{span}\{\1\}\).  Since \(u\perp w\),
their largest principal angle is \(\pi/2\) and their Grassmann distance is
one.

For orthonormal vectors \(a,b\),
\(\langle aa^\top,bb^\top\rangle_F=(a^\top b)^2\).  The four rank-one
projectors in \(P_{0,\tau}-P_{1,\tau}\) are therefore Frobenius-orthogonal,
so
\[
 \norm{P_{0,\tau}-P_{1,\tau}}_F^2=4\tau^2.
\tag{C.1}\label{eq:collision-frob}
\]
If all off-diagonal probabilities lie in \([\eta,1-\eta]\), the elementary
Bernoulli bound
\[
 \KL\{\operatorname{Bern}(a),\operatorname{Bern}(b)\}
 \le\frac{(a-b)^2}{\eta(1-\eta)}
\]
and product additivity give
\[
 \KL(\mathbb P_{0,\tau},\mathbb P_{1,\tau})
 \le \frac{2\tau^2}{\eta(1-\eta)}.
\tag{C.2}\label{eq:collision-kl}
\]
The factor two rather than four appears because each undirected edge is
counted once.  Pinsker's inequality yields
\(\TV(\mathbb P_{0,\tau},\mathbb P_{1,\tau})
\le\tau/\sqrt{\eta(1-\eta)}\), completing the explicit \(O(\tau)\)
calculation used in Section~\ref{sec:lower}.

\section{Reproducibility protocol and additional diagnostics}
\label{app:repro}

The supplied program performs the following steps.
\begin{enumerate}[leftmargin=18pt]
\item For each \((n,p,q)\), compute \(v(P)\), \(g_2(P)\), the raw
matrix-Bernstein radius and the ideal exact-\(k\)-means model-envelope bound
before generating any graph.
\item Generate a symmetric Bernoulli adjacency matrix from its upper
triangle, with a fixed independent random-number stream for every simulation
cell.
\item Compute the two algebraically largest eigenvectors.  Evaluate the
Grassmann error through the smallest singular value of
\(U_\star^\top\widehat U\), avoiding formation of two dense projectors:
\[
\dgr(\widehat U,U_\star)
=\sqrt{1-\sigma_{\min}^2(U_\star^\top\widehat U)}.
\]
\item Apply descriptive Lloyd \(k\)-means to the rows of \(\widehat U\), with
ten initialisations, and match labels by the Hungarian assignment algorithm.
No approximation factor is attributed to this heuristic.
\item Save a checkpoint after every cell, followed by the final summary CSV
files.  Figures read those saved summaries.
\end{enumerate}

The nested bootstrap experiment reuses the same observed eigenspace in both
comparators.  The plug-in SBM method estimates the two within/between
probabilities from the spectral labels.  The simplified ASE-\(k\)NN method scales the
observed eigenvectors by the square roots of their nonnegative eigenvalues,
finds \(20\) nearest neighbours for each vertex, averages their adjacency
rows, symmetrises and clips the resulting probability estimate, and draws
independent Bernoulli graphs.  The bootstrap radius is the empirical
conditional \(0.95\) quantile of the displacement from the observed
eigenspace.  These radii are empirical comparators, not substituted into
Theorem~\ref{thm:grassmann}.  The comparator does not implement the
exchangeability validation step of the full method cited in the text.

\begin{table}[ht]
\centering
\caption{Outputs supplied for independent checking.}
\begin{tabular}{>{\raggedright\arraybackslash}p{0.38\textwidth}
>{\raggedright\arraybackslash}p{0.53\textwidth}}
\toprule
File&Contents\\
\midrule
\path{anc/results/simulation_summary.csv}&Twelve-cell coverage, radius, error and
clustering summary\\
\path{anc/results/operating_envelope_curve.csv}&Radius and quotient-Hamming operating curves
over \(n\)\\
\path{anc/results/operating_envelope_grid.csv}&Crossing sample sizes over the
\((p,q)\) grid\\
\path{anc/results/centrality_simulation_summary.csv}&Degree-band coverage, certificate
firing and selection accuracy\\
\path{anc/results/bootstrap_comparison_raw.csv}&All outer-replication bootstrap
radii and coverage indicators\\
\path{anc/results/karate_application.json}&Every numerical input and diagnostic for
the real-network analysis\\
\path{anc/scripts/run_analysis.py}&Complete simulation, application and figure code\\
\bottomrule
\end{tabular}
\label{tab:files}
\end{table}

The script records the master seed and software versions.  The full
\(12{,}000\)-graph study, \(1000\)-graph centrality study and nested bootstrap
are run from the package root by
\begin{verbatim}
python anc/scripts/run_analysis.py --replications 1000 \
  --centrality-replications 1000 \
  --bootstrap-outer 100 --bootstrap-inner 100
\end{verbatim}
No numerical value is manually entered into a figure.

\section*{Data and code availability}

The arXiv ancillary files include the complete Python program, fixed random
seed, generated CSV outputs and the code that generates every figure and
table.  An identical archive will also be deposited in a permanent public
repository before submission; the package includes a short pre-submission
checklist so that the final repository URL is added to the manuscript and
cover letter rather than fabricated here.
The Zachary karate-club graph is distributed through the standard NetworkX
dataset implementation and is originally documented by
\citet{Zachary1977KarateClub}.  No confidential data are used.

Version 1 of this article corresponds to the rejected manuscript and remains
available at
\href{https://arxiv.org/abs/2602.10566v1}{arXiv:2602.10566v1}.
The present version is a complete reconstruction with corrected finite-sample
constants, nonvacuity analysis, computable certificates and empirical
validation.

\section*{Acknowledgements}

The authors gratefully acknowledge the staff and students of the Department of Mathematics, Government College (Autonomous), Rajahmundry, for their valuable advice and helpful discussions during the preparation of this article. The authors also sincerely thank the Principal of the College for continued encouragement and moral support throughout this research.

\section*{Conflict of interest}

The authors declare no conflict of interest.

\section*{Funding}

The authors received no external funding for this work.

\section*{Ethical approval}

The study uses simulated graphs and a publicly available historical network.
No new data were collected from human participants or animals.

\bibliographystyle{plainnat}
\bibliography{ref}

\end{document}